\documentclass[11pt]{amsart}

\usepackage{pstricks,pst-plot,}
\usepackage{color}
\usepackage{cancel}
\usepackage{tikz}
\usepackage{float}
\usepackage{enumitem}
\usepackage{cancel}
\usepackage[document]{ragged2e}

\usepackage[colorinlistoftodos]{todonotes}

\makeatletter
\@namedef{subjclassname@2020}{%
  \textup{2020} Mathematics Subject Classification}
\makeatother

\definecolor{Black}{cmyk}{0,0,0,1}
\definecolor{OrangeRed}{cmyk}{0,0.6,1,0} 
\definecolor{DarkBlue}{cmyk}{1,1,0,0.20}
\definecolor{myblue}{rgb}{0.66,0.78,1.00}
\definecolor{Violet}{cmyk}{0.79,0.88,0,0}
\definecolor{Lavender}{cmyk}{0,0.48,0,0}

\newcommand\spiral{}
\def\spiral[#1](#2)(#3:#4:#5){
	\pgfmathsetmacro{\domain}{pi*#3/180+#4*2*pi}
	\draw [#1,shift={(#2)}, domain=0:\domain,variable=\t,smooth,samples=int(\domain/0.08)] plot ({\t r}: {#5*\t/\domain})
}
\parskip=\smallskipamount

\newtheorem{theorem}{Theorem}[section]
\newtheorem{lemma}[theorem]{Lemma}
\newtheorem{corollary}[theorem]{Corollary}
\newtheorem{proposition}[theorem]{Proposition}

\theoremstyle{definition}
\newtheorem{definition}[theorem]{Definition}

\newtheorem{remark}[theorem]{Remark}

\newcommand{\bea}{\begin{eqnarray*}}
\newcommand{\eea}{\end{eqnarray*}}

\numberwithin{equation}{section}

\begin{document}
\title[Increasing sequences of complex manifolds]{Increasing sequences of complex manifolds  with uniform squeezing constants and their Bergman spaces}

\thanks{}
\author[Forn\ae ss and Pal]{John Erik Forn\ae ss and Ratna Pal}

\address{John Erik Forn\ae ss: Department of Mathematical Sciences, NTNU Trondheim, Norway}
\email{john.fornass@ntnu.no}

\address{Ratna Pal: Indian Institute of Science Education and Research Mohali, Knowledge City, Sector -81, Mohali, Punjab-140306, India}
\email{ratna.math@gmail.com, ratnapal@iisermohali.ac.in}

\begin{abstract}
For $d\geq 2$, we discuss $d$-dimensional complex manifolds $M$ that are the increasing union of bounded open sets $M_n$'s of $\mathbb{C}^d$ with a common uniform squeezing constant.  The description of $M$ is given in terms of the corank of the infinitesimal Kobayashi metric of $M$, which is shown to be  identically constant on $M$. The main result of this article says that if $M$ has full Kobayashi corank, then $M$ can be written as an increasing union of the unit ball; if $M$ has zero Kobayashi corank, then $M$ has a bounded realization with a uniform squeezing constant; and if $M$ has an intermediate Kobayashi corank, then $M$ has a local weak vector bundle structure. The above description of $M$ is used to show that the dimension of the Bergman space of $M \subseteq \mathbb{C}^d$ is either zero or infinity. This settles Wiegerinck's conjecture for those pseudoconvex domains in higher dimensions that are increasing union of bounded domains with a common uniform squeezing constant. 
\end{abstract}


\maketitle
\section{Introduction}
In this article we study the {\it union problem}: For $d\geq 2$, let $\{M_n\}$ be a sequence of $d$-dimensional complex manifolds such that $M_n \subset \subset M_{n+1}$, for all $n\geq 1$ and let $M=\bigcup_{n=1}^\infty M_n$. Then the union problem asks whether  $M$ can be described in terms of its exhausting  manifolds  $M_n$'s.  The problem goes back to the classical Levi Problem. This can be stated as the question of whether a domain $\Omega$ is Stein if it is locally Stein in the sense that
for every boundary point $p$ there is a neighbourhood $V_p$ so that $\Omega \cap V_p$ is Stein. Another way to interpret the notion of local is to assume that for every compact set $K\subseteq \Omega$ there is a neighbourhood which is Stein. Then one asks whether $\Omega$ is Stein. The answer to this question is affirmative in case $\Omega$ is an open subset of $\mathbb{C}^d$. In other words, if $\Omega$ is an open subset of $\mathbb{C}^d$ and if $\Omega$ is an increasing union of Stein open sets, then $\Omega$ is Stein (\cite{BS}). However, this result is no longer true when $\Omega$ is an arbitrary complex manifold.  In \cite{F}, the first author found  example of a complex manifold which is an increasing union of unit ball but fails to be Stein. 

In the present article, we demonstrate how the notion of {\it squeezing functions} can be employed effectively to explore certain aspects of the union problem. The notion of squeezing function was introduced by  Deng, Guan and, Zhang (\cite{DGZ})  almost a decade back.
Squeezing function of a given domain measures the largest Euclidean ball contained in the injective holomorphic images of the domain and one of its main use is to study the metric geometry of the underlying domain. Let us recall the definition of squeezing function.  In the following definition of squeezing function  and later in this article,  for $\rho>0$, $B(0;\rho)= \{z\in \mathbb{C}^d: \lVert z\rVert <\rho\}$.
\begin{definition}
Let $\Omega \subseteq \mathbb{C}^d$ be a bounded domain. For a point $p\in \Omega$ and a holomorphic embedding $f:\Omega \rightarrow B(0;1)$, we define
\[
S_\Omega(p;f)=\sup \{\rho: B(0;\rho)\subseteq f(\Omega)\},
\]
and the squeezing function at $p$, 
\[
S_\Omega(p)=\sup \{S_\Omega(p;f): f:\Omega \rightarrow B(0;1) \text{ holomorphic embedding} \}.
\]

\end{definition}
It follows immediately from the definition that  the squeezing functions are biholomorphic invariant.  Further they are positive and always bounded above by 1. A domain $\Omega$ is said to have a uniform squeezing constant if there exists an $r>0$ such that $S_\Omega(p)>r$, for all $p\in \Omega$.  Bounded domains with uniform squeezing constants are precisely the holomorphic homogeneous regular domains, introduced and christened by Liu  et al. (\cite{LSY1}). Many interesting bounded domains fall in this class such as homogeneous domains, bounded domains covering compact K\"{a}hler  manifolds, convex domains, strictly pseudocovex domains with $C^2$-smooth boundaries.

In the present article, we assume that the exhausting manifolds $M_n$'s are bounded open sets in $\mathbb{C}^d$ with a common positive uniform squeezing constant.  In particular, $M_n$'s are not necessarily biholomorphic to a fixed bounded domain. This is in contrast to the existing body of work on the union problem where one typically assumes the exhausting domains to be biholomorphic to a fixed domain. In one of the earlier work on this theme, the first author and Stout in \cite{FSt} showed that if $M$ is taut and if each $M_n$ is biholomorphic to the polydisc, then $M$ is also biholomorphic to the polydisc.  Later this result was generalized substantially by the first author and Sibony in \cite{FS}.  The main result of \cite{FS} (to be discussed briefly in the next paragraph) was then obtained for $C^2$-smooth strictly pseudoconvex bounded domains by Behrens \cite{Beh}. The recent papers  \cite{BBMV} by Balakumar et al. and  \cite{ThuVu} by Thu et al. describe the increasing unions of several geometrically interesting domains under the assumption that the final union $M$ is hyperbolic. Exhausting domains considered in the papers \cite{BBMV}, \cite{Beh},  and  \cite{ThuVu} are amenable to {\it scaling methods}, which plays the key role in obtaining structures of the final union $M$.

In spirit, our main result aligns with the one in \cite{FS}. We briefly discuss the main result of \cite{FS}.  Suppose $M_n$'s are biholomorphic to a fixed manifold $\Omega$ with $\Omega/ {\rm{Aut}(\Omega)}$ compact.  Here ${\rm{Aut}(\Omega)}$ denotes the automorphism group of $\Omega$ with the standard compact open topology. The union manifold $M$ is described in terms of the {\it infinitesimal Kobayashi metric} of $M$ (see Definition \ref{Kobayashi metric}), more precisely in terms of the corank (see Remark \ref{corank}) of the Kobayashi metric. It turns out that the corank is identically constant on $M$ in this case. Further the following holds: if the corank is zero then $M$ is biholomorphic to $\Omega$;  if the corank is one, then $M$ is locally trivial holomorphic fiber bundle over $A$ with fiber $\mathbb{C}$, where $A$ is a complex closed submanifold of $\Omega$; structure of $M$ may be very complicated if the corank is strictly greater than one. Note that the exhausting domains considered in \cite{FS}, if bounded, are regular holomorphic homogeneous, i.e., they  have uniform squeezing constants.  Thus the complex manifolds $M$ considered in \cite{FS} are subsumed by the one considered in this article when $\Omega$ is bounded. However, as in \cite{FS}, the Kobayashi metric plays a crucial role in this work.  Let us recall the definition of infinitesimal Kobayashi  metric and Kobabayashi pseudodistance. 

\begin{definition}\label{Kobayashi metric}
Let $M$ be a complex manifold. For $p\in M$ and $\xi$ a tangent vector of $M$ at $p$,  the infinitesimal Kobayashi metric is defined as follow:
\[
K_M(p,\xi)=\inf \{r>0: \rho: \Delta \rightarrow M \text{ holomorphic with }\rho(0)=p, \rho'(0)=\xi/r\},
\]
where $\Delta$ is the unit disc in $\mathbb{C}$. 
\end{definition}

The Kobayashi pseudodistance $d_K^M: M\times M \rightarrow \mathbb{R}$ is defined as follows: 
\[
d_K^M(p,q)=\inf_{\gamma} \int_0^1 K_M(\gamma(t),\gamma'(t)) dt,
\]
where $\gamma$ is a smooth curve joining $p,q\in M$.

Before going further, let us present two definitions pertinent to our main theorem.  In the following definitions, for $0\leq k <d$,  $B^{d-k}(0;1)$ denotes the $(d-k)$-dimensional unit ball.
\begin{definition}\label{weak vector bundle}
Let $M$ be a $d$-dimensional manifold and let $Z$ be a $k$ dimensional manifold where $1\leq k<d$.  Then $M$ is said to have a {\it weak vector bundle structure} of rank $(d-k)$ over $Z$ if there exists a holomorphic function $\psi: M \rightarrow Z$ such that for every $q_0\in Z$,  there exists an open neighbourhood $U$ of $q_0$ in $Z$ and a sequence of injective holomorphic maps $\lambda_n: U\times B^{d-k}(0;1) \rightarrow M_n \subseteq M$  satisfying
\begin{itemize}
\item[(I)] 
$\lambda_n\left(\{q\}\times {B}^{d-k}(0;1)\right) \subseteq  \lambda_{n+1}(\{q\}\times {B}^{d-k}(0;1))$
for all $n\geq 1$ and $\psi^{-1}\{q\}= \bigcup_{n=1}^\infty \lambda_n \left(\{q\} \times {B}^{d-k}(0;1)\right)$ for all $q\in U$;
\item[(II)] 
$\lambda_n(U\times {B}^{d-k}(0;1)) \subseteq  \lambda_{n+1}(U\times {B}^{d-k}(0;1))$
for all $n\geq 1$  and $\psi^{-1}(U)=\bigcup_{n=1}^\infty \lambda_n(U\times {B}^{d-k}(0;1))$;
\item[(III)]
$\psi \circ \lambda_n(q,w)=q$ for all $(q,w)\in U\times {B}^{d-k}(0;1)$.
\end{itemize}
\end{definition}
Definition \ref{weak vector bundle} is inspired by Lemma 4.5 in \cite{FS} which  shows that the union manifolds considered in \cite{FS} have  weak vector bundle  structures when the corank of the infinitesimal Kobayashi metric is strictly greater than one and strictly less than $d$. Definition \ref{weak local vector bundle}, a local version of Definition \ref{weak vector bundle}, is formulated in view of the structures of the union manifolds  we obtain in the present article under the same  assumption on the corank of the union manifolds as described above.
\begin{definition}\label{weak local vector bundle}
Let $M$ be a $d$-dimensional manifold and let $1\leq k<d$. Then $M$ is said to have a {\it local weak vector bundle structure} of rank $(d-k)$ if the following holds.  For each $p\in M$, there exists a holomorphic function
$\psi: M \rightarrow B(0;1)$ with $\psi(p)=0$ and a $k$-dimensional submanifold $Z_{\psi} \subseteq B(0;1)$ such that for every $q_0\in Z_{\psi}$, there exists a sufficiently small neighbourhood $U$ of $q_0$ in $Z_{\psi}$ and for each $n\geq 1$, there exists a injective holomorphic maps $\lambda_n: U\times B^{d-k}(0;1) \rightarrow M_n \subseteq M$ satisfying 

\begin{itemize}
\item[(I)] 
$\lambda_n\left(\{q\}\times {B}^{d-k}(0;1)\right) \subseteq  \lambda_{n+1}(\{q\}\times {B}^{d-k}(0;1))$
for all $n\geq 1$ and $\psi^{-1}\{q\}= \bigcup_{n=1}^\infty \lambda_m \left(\{q\} \times {B}^{d-k}(0;1)\right)$ for all $q\in U$;
\item[(II)] 
$\lambda_n(U\times {B}^{d-k}(0;1)) \subseteq  \lambda_{n+1}(U\times {B}^{d-k}(0;1))$
for all $n\geq 1$  and $\psi^{-1}(U)=\bigcup_{n=1}^\infty \lambda_n(U\times {B}^{d-k}(0;1))$;
\item[(III)]
$\psi \circ \lambda_n(q,w)=q$ for all $(q,w)\in U\times {B}^{d-k}(0;1)$.
\end{itemize}

\end{definition}
The main result of this article is as follows.
\begin{theorem}\label{main thm}
 For $d\geq 2$ and for each $n\geq 1$, let $M_n$ be biholomorphic to (possibly different) bounded open sets in $\mathbb{C}^d$ such that $M_n\subset \subset M_{n+1}$  and suppose $M=\bigcup_{n=1}^\infty M_n$. Suppose the family $\{M_n\}_{n\geq1}$ has a uniform squeezing constant, say $r>0$. Then for each point $p\in M$, the set $\mathcal{K}_p=\{\xi: K_M(p;\xi)=0\}$ forms a  vector space and the correspondence $p\mapsto k_p$, where $k_p$ is the dimension of $\mathcal{K}_p$, is identically constant on $M$. Further the following holds:
\begin{itemize}
\item[(I)] 
If $k_p\equiv d$ on $M$, then $M$ can be written as an increasing union of $d$-dimensional unit ball $B(0;1)$. Conversely, if $M$ is an increasing union of unit ball and if $\Omega$ is \textbf{any} bounded domain with a uniform squeezing constant,  then $M$ can be written as an increasing union of $\Omega$. Thus if $\Omega_1$ and $\Omega_2$ are two bounded domains with uniform squeezing constants $r_1$ and $r_2$, respectively and if $M$ is an increasing union of $\Omega_1$, then $M$ can be written as an increasing union of $\Omega_2$ and vice versa;
\item[(II)] 
If $k_p\equiv 0$ on $M$, then there exists an injective embedding of $M$ into the unit ball $B(0;1)$. Further, $r$ is a uniform squeezing constant of $M$;
\item[(III)] 
If $0<k_p\equiv (d-k)<d$ on $M$, then $M$ has a local weak vector bundle structure of rank $d-k$ (see Definition \ref{weak local vector bundle}).
\end{itemize}
\end{theorem}
A couple of remarks are in order.
\begin{remark} \label{corank}
The natural number $k_p$ at the point $p\in M$ is called the {\it Kobayashi corank} of $M$ at the point $p$. The number $d-k_p$ is called the {\it Kobayashi rank} of $M$ at the point $p$.
\end{remark}
\begin{remark}
Manifolds that are increasing union of balls and on which the Kobayashi metric vanishes identically is of particular interest.  Till date only two such mutually non-biholomorphic classes of manifolds are identified: Fatou-Bieberbach domains and Short $\mathbb{C}^d$'s. The existence of other classes of manifolds in this category is strongly believed and identifying these manifolds has been a long time pursuit. The part (I) of Theorem \ref{main thm} attempts to give a flexibility result. Manifolds on which Kobayashi metric does not vanish identically may not enjoy this sort of flexibility under the same assumption. For example (see \cite{FS}), if $M$ is an increasing union of balls or polydiscs, then $M$ is biholomorphic to ball or polydiscs, respectively, when $M$ is hyperbolic, i.e., when $k_p\equiv 0$ on $M$;  $M$ is biholomorphic to a cylinder with fiber $\mathbb{C}$ and base $(d-1)$-dimensional ball or $(d-1)$-dimensional polydisc, respectively, when Kobayashi corank of $M$ is one, i.e., $k_p\equiv 1$ on $M$. 
\end{remark}
Now we discuss two interesting corollaries of Theorem \ref{main thm}. 
The first one is a result of Behrens (\cite{Beh}). Behrens'  proof relies on techniques of \cite{FS} and Pinchuk's scaling method. We give an alternative proof of her result combining techniques of \cite{FS} and properties of squeezing functions. 
\begin{corollary}\label{Behrens}
Let $d\geq 2$. Let $M$ be a $d$-dimensional manifold and let $M$ be an increasing union of a fixed bounded $C^2$-smooth strictly pseudoconvex domain $\Omega$.  Then for each point $p\in M$, the set $\mathcal{K}_p=\{\xi: K_M(p;\xi)=0\}$ forms a  vector space and the function $p\mapsto k_p$, where $k_p$ is the dimension of $\mathcal{K}_p$, is identically constant on $M$. Moreover, 

\begin{itemize}
\item[(I)] 
If $k_p\equiv 0$ on $M$, then $M$ is either biholomorphic to $\Omega$ or $M$ is biholomorphic to the unit ball $B(0;1)$;
\item[(II)] 
If $0<k_p\equiv (d-k)<d$ on $M$, then $M$ has a  weak vector bundle structure of rank $k$ either over a submanifold $A$ of $\Omega$ or over a submanifold $A$ of the unit ball $B(0;1)$.
\end{itemize}
\end{corollary}

\begin{remark}
The same conclusion as of Corollary \ref{Behrens} can be drawn if each $M_n$ is biholomorphic to a fixed bounded domain $\Omega$ and the squeezing function of $\Omega$ approaches to $1$ near boundary. 
\end{remark}

The collection of square integrable holomorphic functions of a domain $D$ is called the Bergman space of $D$ and we denote it by $L^2{(D)}$.  In \cite{Wieg}, Wiegerinck showed  that the dimension of Bergman space of any planar domain and thus of any pseudoconvex domain in $\mathbb{C}$ is either zero or infinity. Whether this dichotomy continues to hold in higher dimensions for pseudoconvex domains is still unsettled. However in higher dimensions examples of nonpseudoconvex  domains with finite dimensional Bergman spaces are known.  A considerable amount of effort has been put to identify several classes of pseudoconvex domains with the dimensions of Bergman spaces either zero or infinity.  The following corollary identifies yet another class of domains in higher dimensions with zero or infinite dimensional Bergman spaces. 
\begin{corollary}\label{main cor}
Let $M\subseteq \mathbb{C}^d$ be an increasing union of bounded domain in $\mathbb{C}^d$ with a fixed common positive squeezing constant. Then the dimension of the Bergman space 
$$L^2(M)=\{f: M\rightarrow \mathbb{C} \text{ holomorphic }: \int_M {\lvert f\rvert}^2 <\infty\}$$
of $M$ is either zero or infinity. Further, in case the corank of the infinitesimal Kobayashi metric of $M$ is strictly greater than zero,  the Euclidean volume of $M$ is infinity.  
\end{corollary}


\section{$M$ with full Kobayashi corank at some point}
In this section we prove the first part of Theorem \ref{main thm}. 
\subsection{Proof of [(I), Theorem \ref{main thm}]} We complete the proof in three steps.


\medskip 
\noindent
{\it Step 1:} In this step we prove that if there exists a point $p\in M$ for which $K_M(p;\xi)=0$, for all 
$\xi \in \mathbb{C}^k$, then the Kobayashi metric vanishes identically on $M$.  Without loss of generality, we assume $p$ to be the origin. Let $L>>1$. Then for $n$ sufficiently large and for any tangent vector $\xi$ of length one, there is a holomorphic map $g_{\xi}:\Delta\rightarrow M_n$
so that $g_{\xi}(0)=0, g_{\xi}'(0)=L\xi$. Here $\Delta$ is the unit disc in $\mathbb{C}$.  
 
Since each $M_n$ has fixed squeezing constant $r>0$, for each $n\geq 1$, there exists injective holomorphic map $\psi_n:M_n\rightarrow B(0;1)$ such that $\psi_n(0)=0$ and $B(0;r) \subseteq \psi_n(M_n)$.  Pick a tangent vector $\xi$ at $0\in M$. Consider the composition $\psi_n\circ g_\xi:\Delta \rightarrow B(0;1)$. By the Schwarz Lemma
the derivative $\|\psi'_n(0)(L\xi)\|\leq 1.$ Hence $\|\psi_n'(0)(\xi)\|\leq 1/L$.  Without loss of generality,  we assume that $\psi_n$ converges uniformly on compacts of $M$ to $\psi$. Thus $\psi'(0)=0$.

{\it Claim:} For any set $K \subset  \subset M$ and  $0<\delta<1$, 
$\psi_n(K)\subseteq B(0,\delta)$, for all $n$ sufficiently large.

The claim follows if we we prove  that $\psi$ is identically constant.  Let us assume on contrary that $\psi$ is non-constant.  Without loss of generality,  we can assume 
$$
\psi(z,0')=(a_1z^{l_1}+\cdots, a_2 z^{l_2}+\cdots, a_d z^{l_d}+\cdots),
$$ 
for $0'=(0,\ldots,0)\in \mathbb{C}^{d-1}$ and for $z\in \mathbb{C}$ with $\lvert z \rvert<\epsilon$, where  $\epsilon>0$ sufficiently small. We can also assume that  $2\le l_1,\ldots, l_d < \infty$ and $a_1 \neq 0.$ Pick arbitrarily large $L>0$ and a map $f$ from $\Delta$ to $M$ so that $f(0)=p$ and $f'(0)=(L,0')$. Then $\psi\circ f(t) = (a_1L^{l_1}t^{l_1}+\cdots, \dots)$.
Since by generalized Schwarz-Pick lemma, it follows that the modulus of $a_1L^{l_1}$ must be at most $1$ (see \cite [Proposition 1.1.2] {JP}), we get a contradiction when $L$ is sufficiently large. Thus $\psi \equiv 0$ on $M$. Therefore the claim follows. 

Fix a $p\in M$ and a $\xi \in \mathbb{C}^d$. Now note that 
\begin{eqnarray*}
K_M(p;\xi)&=&\lim_{n\rightarrow \infty} K_{M_n}(p;\xi)\\
&=&\lim_{n\rightarrow \infty} K_{\psi_n(M_n)}(\psi_n(p);\psi_n'(p)(\xi))\\
&\leq& \lim_{n\rightarrow \infty} K_{B(0;r)}(\psi_n(p);\psi_n'(p)(\xi)) =0,
\end{eqnarray*}
as $n\rightarrow \infty$ and the last equality holds since $\psi_n(p), \psi_n'(p)(\xi)\rightarrow 0$ as $n\rightarrow \infty$. This shows that the Kobayashi metric vanishes identically on $M$.

{\it Step 2: } In this step we prove that $M$ can be written as an increasing union of unit balls. Recall that $0\in M_n$ for all $n\geq 1$. Now for each fixed $n$, we can find an $L_n>0$ such that for any $z\in M_n$ there exists a curve $\gamma_z \subseteq M_n$ joining $0$ and $z$ with $\lVert \gamma_z'(t)\rVert<L_n$ for $0\leq t \leq 1$. Therefore, by replacing $M_n$'s possibly by a subsequence, we can assume that the infinitesimal Kobayashi distance of $M_{n+1}$ restricted to $M_n\times B(0;L_n)$ is smaller than $\epsilon_n  \downarrow 0$. Thus without loss of generality, we can assume that the  Kobayashi distance  between the origin and any point $z\in M_n$ is smaller than $\epsilon_n$, i.e., $d_{K}^{M_{n+1}}(0,z)<\epsilon_n$, for all $z\in M_n$ and for all $n\geq 1$. Here $d_{K}^{M_{n+1}}$ denotes the Kobayashi distance on $M_{n+1}$. 


Recall that for each $n\geq 1$,  $r>0$ is the uniform squeezing constant of $M_n$.  Thus there exists injective holomorphic map $\psi_n:M_n \rightarrow B(0;1)$ such that $B(0;r) \subseteq \psi_n(M_n)$ with $\psi_n(0)=0$.  Therefore 
\begin{eqnarray} \label{squ}
\lVert \psi_{n+1}(z)\rVert \leq d_K^{B(0;1)}(\psi_{n+1}(z),0)\leq d_K^{M_{n+1}}(z,0)<\epsilon_n<r,
\end{eqnarray}
for all $z\in M_n$ and possibly for $n$ large enough. Thus 
$$M_n \subseteq \psi_{n+1}^{-1}(B(0;r)) \subseteq M_{n+1}.$$
Thus we can write $M$ as an increasing union of Euclidean balls. 

{\it Step 3:} 
In this step we prove that if $M$ is an increasing union of unit balls in $\mathbb{C}^d$ and if  $\Omega\subseteq \mathbb{C}^d$ is  a  bounded domain, then $M$ can be written as an increasing union of $\Omega$. 

Let $M=\bigcup_{n=1}^\infty U_n$ with $U_n \subset \subset U_{n+1}$, where each $U_n$ is biholomorphic to the unit ball $B(0;1)$. Without loss of generality, let us assume that $0\in U_n$ for all $n\geq 1$. Let $\psi_n: B(0;1) \rightarrow U_n$ be a biholomorphism with $\psi_n(0)=0$.  Let $\epsilon_n\downarrow 0$, then as in (\ref{squ}), we have
\begin{equation} \label{Omega}
\lVert \psi_{n+1}^{-1}(z)\rVert\leq K_{B(0;1)}(\psi_{n+1}^{-1}(z),0)=K_{U_{n+1}}(z,0)<\epsilon_n
\end{equation}
for all $z\in U_n$.  
Fix a point $p\in \Omega$ and let $S_\Omega(p)=r$. Then there exists an injective holomorphic 
function $\psi: \Omega \rightarrow B(0;1)$ with $\psi(p)=0$ and $B(0;r)\subseteq \psi(\Omega)$. Therefore, by (\ref{Omega}), $\psi_{n+1}^{-1}(U_n) \subseteq B(0;r) \subseteq \psi(\Omega)$, i.e.,  $U_{n} \subseteq (\psi_{n+1}\circ \psi)(\Omega)$. 
Next we prove 
\[
(\psi_{n+1}\circ \psi)(\Omega) \subseteq (\psi_{n+2}\circ \psi)(\Omega)
\]
for all $n$, which is equivalent to showing that for all $n\geq 1$, 
\[
(\psi_{n+2}^{-1}\circ \psi_{n+1})(\psi(\Omega) )\subseteq  \psi(\Omega).
\]
Note that the above follows from (\ref{Omega}). Thus we are done.


\section{$M$ with zero Kobayashi corank at some point}
In this section we prove the second part of Theorem \ref{main thm}

\subsection{Proof of [(II), Theorem \ref{main thm}]}
We complete the proof in two steps. 

\medskip 
\noindent
{\it Step 1:} In this step we prove that the Kobayashi rank  is identically $d$ on $M$.   Let  $p\in M$ be such that $K_M(p;\xi)\neq 0$ for all $\xi\neq 0$.  Let for each $n\geq 1$,  $\psi_n: M_n \rightarrow B(0;1)$ be an injective mapping with $\psi_n(p)=0$ and $B(0;r) \subseteq \psi_n(\Omega)$. Since $B(0;1)$ is taut,  we can assume that $\psi=\lim_{n\rightarrow \infty} \psi_n: M\rightarrow B(0;1)$ defines a holomorphic map with $\psi(p)=0$. Now observe the following. 

\begin{multline*}
K_{B(0;1)}\left(0;\psi_n'(p)(\xi)\right)\leq K_{\psi_n(M_n)}\left(0, \psi_n'(p)(\xi)\right)=K_{M_n}\left(p, \xi\right)\\
 \leq K_{B(0;r)}\left(0;\psi_n'(p)(\xi)\right).
\end{multline*}
Thus  
\begin{multline*}
\lim_{n\rightarrow \infty} K_{B(0;1)}\left(0; \psi_n'(p)(\xi)\right)\leq \lim_{n\rightarrow \infty}K_{\psi_n(M_n)}\left(0, \psi_n'(p)(\xi)\right) 
=\lim_{n\rightarrow \infty} K_{M_n}(p, \xi)\\
\leq \lim_{n\rightarrow \infty} K_{B(0;r)}\left(0; \psi_n'(p)(\xi)\right),
\end{multline*}
 which in turn gives
\begin{multline*}
 \lVert \psi'(p)(\xi)\rVert = K_{B(0;1)}(0;\psi'(p)(\xi))\leq 
  K_M(p;\xi) \\
\leq  K_{B(0;r)}(0;\psi'(p)(\xi))=\frac{1}{r}  \lVert \psi'(p)(\xi)\rVert.
\end{multline*}
Therefore, $\psi'(p)$ has full rank at $p$ and thus  $\det [\psi'(p)] \neq 0$. By Hurwitz's theorem $\psi'$ has full rank at each point of $M$. Thus $\psi$ is non-constant. Moreover, running the above construction for each $z\in M$, we get that $K_M(z;\xi)\neq 0$ for all $z\in M$ and for all $\xi \neq 0$. Therefore the Kobayashi rank is identically $d$ on $M$.

{\it Step 2:} In this step we prove that $\psi$ is injective, which in turn gives that $M$ has a bounded realization. Then by Thm 2.1 in \cite{DGZ1}, we get  
$$
r\leq \lim_{n\rightarrow\infty}S_{M_n}(z)=S_M(z),
$$ 
for all $z\in M$. Thus upto biholomorphism $M$ is  a bounded regular homogeneous domain. In particular, $M$ is Kobayashi complete. 

Let us assume that  $\psi$ is not injective. Then there exist $z_1, z_2\in M$ such that $\psi(z_1)=\psi(z_2)$. Now since $\psi'$ has full rank at each point of $M$, there is a small ball $\mathcal{B}_{z_1}$ with centre at $z_1$ such that $\psi$ is injective in a slightly larger open set $\mathcal{U}$ containing $\mathcal{B}_{z_1}$. Further we can assume that $\psi$ has no zeros on $\partial \mathcal{B}_{z_1}$. 
Let $\mathcal{U} \cup \{z_2\}\subset \subset M_n$,  for all $n\geq n_0$. Define $\tilde{\psi}_n(z)=\psi_n(z)-\psi_n(z_2)$, for $n\geq  n_0$ and for $z\in  \mathcal{U}$. Similarly define $\tilde{\psi}(z)=\psi(z)-\psi(z_2)$, for $z\in  \mathcal{U}$. Now let $0<\delta=\inf_{z\in \partial \mathcal{B}_{z_1}} \lVert \tilde{\psi}(z)\rVert$. Thus
\[
\lVert \tilde{\psi}_n(z)-\tilde{\psi}(z)\rVert <\delta/2< \lVert \tilde{\psi}(z)\rVert, 
\]
for all sufficiently large $n$ and for  all $z\in \partial \mathcal{B}_{z_1}$. Using higher dimensional Rouch\'{e}'s theorem, $\tilde{\psi}_n$ and $\tilde{\psi}$ have the same number of zeros counted with multiplicities in $\mathcal{B}_{z_1}$. Note that $\tilde{\psi}$ has a zero at $z_1$ but none of the $\tilde{\psi}_n$'s has any zero in $\mathcal{B}_{z_1}$. This is a contradiction. Thus $\psi$ is injective on $M$. Consequently, $M$ has a bounded realization.

\section{$M$ with intermediate Kobayashi corank at some point}

In this section we prove the third part of Theorem \ref{main thm} using techniques of \cite{FS} in our set-up. However, recall that unlike in \cite{FS},  the exhausting domains considered in our case are not necessarily biholomorphic to a fixed domain. Further, if bounded, the class of exhausting domains considered here comes from a much larger class than the one considered in \cite{FS}.

 \subsection{ Proof of [(III), Theorem \ref{main thm}]}
We complete the proof in two steps.

\medskip 
\noindent
{\it Step 1:} Let $p\in M$ be such that the Kobayashi corank is $d-k$. Thus the Kobayashi rank is $k$ at $p$.  In this step we prove that the Kobayashi rank  is identically $k$ on $M$. 
For each $n\geq 1$, let $\psi_n: M_n \rightarrow B(0;1)$ be an injective map with $\psi_n(p)=0$ and $B(0;r) \subseteq \psi_n(M_n)$.  Since $B(0;1)$ is taut, without loss of generality,  $\psi=\lim_{n\rightarrow \infty}\psi_n: M\rightarrow B(0;1)$ is a holomorphic function with $\psi(p)=0$. 

\noindent 
Note that
\begin{equation*}
K_{B(0;1)}\left(0;\psi_n'(p)(\xi)\right) \leq  K_{\psi_n(M_n)}(0; \psi_n'(p)(\xi)) \leq K_{B(0;r)}\left(0;\psi_n'(p)(\xi)\right).
\end{equation*}
This implies 
\begin{eqnarray*}
\lim_{n\rightarrow \infty} K_{B(0;1)}(0;\psi_n'(p)(\xi)) &\leq& \lim_{n\rightarrow \infty}K_{\psi_n(M_n)}(0, \psi_n'(p)(\xi)) \\
 &\leq& \lim_{n\rightarrow \infty} K_{B(0;r)}(0; \psi_n'(p)(\xi)).
\end{eqnarray*}
Therefore, 
\begin{align}\label{rank}
& \lVert \psi'(p)(\xi)\rVert=K_{B(0;1)}(0;\psi'(p)(\xi)) \leq \lim_{n\rightarrow \infty}K_{\psi_n(M_n)}(0,  
\psi_n'(p)(\xi)) \nonumber \\
 &= \lim_{n\rightarrow \infty} K_{M_n}(p;\xi) = K_M(p;\xi)  
\leq K_{B(0;r)}(0;\psi'(p)(\xi))=\frac{1}{r}  \lVert \psi'(p)(\xi)\rVert.
\end{align}

Thus  the Kobayashi rank  at $p$ is $k$ if and only if the rank of $\psi'(p)$ is $k$. This shows that  the map $\psi$ is not identically constant on $M$. 

Let $B(0;r)\subseteq \psi_n(M_n)=U_n$ and let $\phi_n:U_n\rightarrow M_n$ is the inverse of $\psi_n$. For each $n\geq 1$, set $\alpha_n=\psi\circ \phi_n.$ We restrict the domain of $\alpha_n$ to
$B(0,r)$ as points outside $B(0,r)$ might not be in the domain of all $\alpha_n$'s.
Thus each $\alpha_n$ maps $B(0,r)$ into $B(0,1)$. 
Without loss of generality, set $\alpha=\lim_{n\rightarrow \infty} \alpha_n$. Again clearly, $\alpha$  maps $B(0,r)$ into $B(0,1)$. 

\begin{lemma} \label{L1}
Let $M_\psi=\psi^{-1}(B(0,r))$. Then $\alpha\circ \psi=\psi$ on $M_\psi$. 
\end{lemma}
\begin{proof}
Suppose that $p\in M_\psi$. Then $\psi_n(p)$ converge to $\psi(p)\in B(0,r).$ Moreover, $\{\alpha_n\}$ converges uniformly on compacts in $B(0,r)$ to $\alpha.$ Hence $\alpha\circ \psi(p)=\lim \alpha_n\circ \psi_n(p).$ Thus 
$\alpha\circ \psi(p)= \lim \psi \circ \phi_n\circ \psi_n(p)=\lim \psi(p)=\psi(p).$
\end{proof}
Let $k$ be the maximal rank of $\psi$ on $M_\psi$. The maximal rank occurs on a dense open set, so $k$ is also the maximal rank of $\psi$ on $M_\psi.$ Hence for each $n$, the rank of $\alpha_n=\psi\circ \phi_n$ on $B(0,r)$ is at most $k$. So the maximal rank of $\alpha$ is at most $k$ on $B(0,r)$ as well.

\begin{lemma} \label{L2}
Let $Z$ be the closed subvariety of $B(0,r)$ given by $Z=\{q\in B(0,r): \alpha(q)=q\}$. Then the dimension of $Z$ is at most $k$. If $Z$ has dimension $k$ at $q$, then $Z$ is a $k$ dimensional complex manifold on a neighbourhood of $q$. 
\end{lemma}
\begin{proof}
The proof follows by using the same argument as in \cite[Lemma 4.2]{FS}. Hence the proof is omitted here. 
\end{proof}

By definition, $M_\psi=\psi^{-1}(B(0,r))$ and thus $\psi(M_\psi)) \subseteq B(0,r)$. By Lemma \ref{L1}, $\alpha\circ \psi=\psi$ on $M_\psi$. Hence $\psi(M_\psi)\subseteq Z$. Let $\tilde{M}_\psi$ be the connected component of $M_\psi$ containing $p$. Let $\tilde{Z}_\psi$ be the connected component of $Z$ containing $\psi(\tilde{M}_\psi)$.  Since $\psi$ has maximal rank $k$, it follows that $\tilde{Z}_\psi$ has dimension at least $k$ at some point.  This implies that $\tilde{Z}_\psi$ has at least one irreducible component of dimension  $k.$  So we can assume that $\psi(\tilde{M}_\psi)$ is contained in an irreducible component of $\tilde{Z}_\psi$ and $\tilde{Z}_\psi$ is a closed complex submanifold of $B(0,r)$ of dimension $k.$

\begin{lemma} 
The map $\psi$ has constant rank $k$ on the connected component $\tilde{M}_\psi$ of $M_\psi$ containing $p$. Further, the rank of the infinitesimal Kobayashi metric  is identically $d-k$ on $\tilde{M}_\psi$. 
\end{lemma}

\begin{proof}
Let $\tilde{Z}_\psi$ be as in the previous proof.  Since $\alpha$ has rank $k$ on $Z_\psi$, there is a neighbourhood $W$ of $\tilde{Z}_\psi$ in $B(0,r)$ on which the rank of $\alpha$ is identically $k$. Fix a point $q\in \tilde{M}_\psi.$ Since $\psi(q)\in \tilde{Z}_\psi,$ there is a neighbourhood $U\subset\subset W$ of $q$ such that $\psi_n(q)\in U$ for all large $n$. The sequence of functions $\alpha_n$ converges uniformly on compact subsets of $B(0,r)$ to $\alpha$. Hence for all large enough $n$, the rank of $\alpha_n$ is at least $k$ for every point in $U$. In particular, $\alpha_n=\psi\circ \phi_n$ has rank at least $k$ at $\psi_n(q).$ Hence $\psi$ has rank at least $k$ at $\phi_n\circ \psi_n(q)$. Since the rank of $\psi$ is at most $k,$ we have shown the first part of the lemma, i.e.,  $\psi$ has constant rank $k$ on $\tilde{M}_\psi.$ Further, it follows from (\ref{rank}) that the Kobayashi rank  is locally constant on $\tilde{M}_\psi$. 
\end{proof}

\noindent
Thus the Kobayashi rank is locally constant on $M$, which in turn gives that the Kobayashi rank is identically constant on the whole $M$. 

\medskip 
\noindent
{\it Step 2:}  Let $\psi(\tilde{M}_\psi)=Z_{\tilde{M}_\psi} \subseteq \tilde{Z}_\psi$. Let $q_0\in Z_{\tilde{M}_\psi}$ and let $\hat{U}$ be a small neighbourhood of $q_0$ in $B(0;r)$.
Note that $U=Z_{\tilde{M}_\psi}\cap \hat{U}$ is a $k$-dimensional submanifold. So there exist 
$f_1, \ldots, f_{d-k}, f_{d-k+1},\ldots,f_d: \hat{U}\rightarrow \mathbb{C}$ (these functions give a local change of coordinate near $q_0$) such that $ Z_{\tilde{M}_\psi} \cap \hat{U}=\{f_1= \ldots= f_{d-k}=0\}$. We write $F_1=(f_1,\dots,f_{d-k}), F_2=(f_{d-k+1},\dots, f_d)$ and $F=(F_1,F_2).$ Then the map $F$ has rank $d$ at $q_0$. We define the map $\theta$ on $\hat{U}$ as follows:
$$
\theta(q)= (\alpha(q),F_1(q)).
$$
We show that $\theta'$ has rank $d$ at $q_0$ and hence that $\theta$ is one to one in a
neighborhood. Note that for any fixed $l_0$, $\alpha^{-1}(l_0)$ is an $d-k$ dimensional submanifold in $\hat{U}$. Also since $\alpha$ is identity on $Z$, $\alpha^{-1}(l_0)$ intersects $Z_{\tilde{M}_\psi} $ at a single point. Consider the function defined on  $\alpha^{-1}(l_0)$ as $q\rightarrow F_1(q)$.
Let $\xi\neq 0$ be a tangent vector to $\hat{U}$ at $q_0.$
If $\xi$ is tangent to 
$Z_{\tilde{M}_\psi} $, then $F_1'(\xi)=0$ so $F_2'(\xi)\neq 0.$ Thus the nullspace $N$
of $F_2'(\xi)$ is $d-k$ dimensional and transverse to $Z_{\tilde{M}_\psi} $.
Hence $F_1'\neq 0$ on $N.$ But $F_1'=0$ on the tangent space of$Z_{\tilde{M}_\psi} $.
So $F_1'\neq 0$ on every vector transverse to the tangent space of$Z_{\tilde{M}_\psi}$,
hence on the null space of $\alpha'.$ Therefore $\theta'$ has rank d.
Further, one can modify $\theta$ (and possibly shrink $\hat{U}$ a bit) as 
$$\theta(q)=(\alpha(q),F_1(q)/\epsilon)$$ so that  $\theta: \hat{U} \rightarrow U\times {B}^{d-k}(0;1)$ is holomorphic one-one and onto. 
The sequence of  maps $\theta_m:q\mapsto (\alpha_m(q),F_1(q)/\epsilon)$ converges uniformly on compacts to the map $\theta$ on $\hat{U}$. Thus for all large $m$, the map $\theta_m$ sends a small neighbourhood $\hat{U}_m$ biholomorphically to $U\times {B}^{d-k}(0;1)$, where ${B}^{d-k}(0;1)$ is the $(d-k)$-dimensional unit ball. Now consider the map 
$$
\lambda_m=\phi_m\circ \theta_m^{-1}:U\times {B}^{d-k}(0;1) \rightarrow M_m \subseteq M. 
$$


Now we observe the following. 

\medskip 
\noindent 
(i) $\psi \circ \lambda_m(q,w)=q$ for all $(q,w)\in U\times {B}^{d-k}(0;1)$.  

\begin{proof}
Let $(q,w)\in U\times {B}^{d-k}(0;1)$. Let $p=\theta_m^{-1}(q,w)$. Then 
\begin{align*}
(q,w)=\theta_m(p)=(\alpha_m(p), F_1(p)/\epsilon)=(\psi \circ \phi_m(p),F_1(p)/\epsilon)\\
=(\psi \circ \phi_m \circ \theta_m^{-1}(q,w),F_1(p)/\epsilon)=\psi \circ \lambda_m(q,w).
\end{align*}
Thus the proof follows. 
\end{proof}

(ii) $\lambda_m(U\times {B}^{d-k}(0;1)) \subseteq  \lambda_{m+1}(U\times {B}^{d-k}(0;1))$
for all $m\geq 1$  and $\psi^{-1}(U)=\bigcup_{m=1}^\infty \lambda_m(U\times {B}^{d-k}(0;1))$.

\begin{proof}
 Let $V_m=\lambda_m (U\times {B}^{d-k}(0;1))\subset \subset M_{m+l}$ for large enough 
$l$.  Now $\{\psi_{m+l}\}_{l\geq 1}$ converges uniformly to $\psi$ on $V_m$. By (i) we have $\psi(V_m) \subseteq U$. Since  $F_1\equiv 0$ on $U$, $\psi_{m+l}(V_m) \subseteq \{\lvert F_1\rvert <\epsilon/m\}$ for large $l$. In particular, $\psi_{m+l}(V_m) \subseteq \hat{U}$ for large  $l$. Thus $\psi_{m+l}(V_m)$ is in the domain of $\theta_{m+l}$. Now note that 
\begin{eqnarray}\label{Lamda m}
&&\theta_{m+l} \circ \psi_{m+l} (\lambda_m(q,w)) \nonumber \\
&=& \left(\alpha_{m+l}\circ \psi_{m+l} (\lambda_m(q,w)), (F_1\circ \psi_{m+l} (\lambda_m(q,w))/\epsilon\right) \nonumber \\
&=& \left(\psi\circ \phi_{m+l}\circ \psi_{m+l} (\lambda_m(q,w)), (F_1\circ \psi_{m+l} (\lambda_m(q,w))/\epsilon\right) \nonumber \\
&=& \left(\psi\circ \lambda_m(q,w), (F_1\circ \psi_{m+l} (\lambda_m(q,w))/\epsilon\right) \subseteq U \times \{\lvert w\rvert < 1/m\}.
\end{eqnarray}
Thus $\theta_{m+l} \circ \psi_{m+l} (\lambda_m(q,w))$ with $(q,w)\in U\times {B}^{d-k}(0;1)$ is in the domain of $\lambda_{m+l}$. Further note that 
\[
\lambda_{m+l}\left(\theta_{m+l} \circ \psi_{m+l} (\lambda_m(q,w)\right)=\phi_{m+l}\circ \theta_{m+l}^{-1}\circ \theta_{m+l} \circ \psi_{m+l} (\lambda_m(q,w))=\lambda_m(q,w).
\]
Now since $\theta_{m+l} \circ \psi_{m+l} (\lambda_m(q,w)) \subseteq U \times \{\lvert w\rvert < 1/m\}\subseteq U\times {B}^{d-k}(0;1)$, we have 
\[
V_m=\lambda_m(U\times {B}^{d-k}(0;1)) \subseteq \lambda_{m+l}(U\times {B}^{d-k}(0;1))=V_{m+l}.
\]
Hence without loss of generality, we get $V_m \subseteq V_{m+1}$. 

Let $p\in \psi^{-1}(U)$. Thus $\theta \circ \psi(p)\in U\times \{0\}$. Then 
$\theta_m \circ \psi_m(p)\in U\times {B}^{d-k}(0;1)$ for all large $m$ and  
thus $\lambda_m \circ  \theta_m \circ \psi_m(p)=p \in V_m$. 
Therefore, 
 \[
\psi^{-1}(U)\subseteq \bigcup_{m=1}^\infty \lambda_m(U\times {B}^{d-k}(0;1)).
\]
That $\lambda_m(U\times {B}^{d-k}(0;1)) \subseteq \psi^{-1}(U)$, for all $m\geq 1,$ follows from (i). Hence the proof follows. 
\end{proof}

\medskip
\noindent 
(iii) $\lambda_{m+1}^{-1}\circ \lambda_m(q,0)=(q,0)$ for all $q\in U$. 

\begin{proof}
Note that $\lambda_m^{-1} \circ \lambda_1(q,0)=(q, g_m(q))$ for all $q\in U$ for some holomorphic function $g_m$ on $U$. We modify $\lambda_m$ by defining $\tilde{\lambda}_m: U\times {B}^{d-k}(0;1) \rightarrow M$ as follows:  $\tilde{\lambda}_m(q,w)=\lambda_m(q,2w+g_m(w))$. Thus $\tilde{\lambda}_m(q,0)=\lambda_m \circ T_m(q,0)$ where $T_m(q,w)=(q,2w+g_m(0))$.  Therefore, 
$\tilde{\lambda}_{m+1}^{-1}\circ \tilde{\lambda}_m(q,0)=(q,0)$ for all $q\in U$. Thus (iii) holds for $\tilde{\lambda}_m$. Note that if we replace $\lambda_m$ by $\tilde{\lambda}_m$, (i), (ii), (iv) still hold. 
Thus there is no harm if we replace $\lambda_m$ by $\tilde{\lambda}_m$ in (i), (ii), (iv).
\end{proof}

\medskip
\noindent 
(iv) $\lambda_m\left(\{q\}\times {B}^{d-k}(0;1)\right) \subseteq  \lambda_{m+1}(\{q\}\times {B}^{d-k}(0;1))$
for all $m\geq 1$ and $\psi^{-1}\{q\}= \bigcup_{m=1}^\infty \lambda_m \left(\{q\} \times {B}^{d-k}(0;1)\right)$ for all $q\in U$.

\begin{proof}
It follows from (i) that $\bigcup_{m=1}^\infty \lambda_m \left(\{q\} \times {B}^{d-k}(0;1)\right)\subseteq \psi^{-1}\{q\}$. Now let $z\in \psi^{-1}(q)$. Then we have $\psi^{-1}(q) \subseteq \bigcup_{m=1}^\infty \lambda_m(U\times {B}^{d-k}(0;1))$ by (ii). Hence it follows from (i) that $z\in \lambda_m \left(\{q\} \times {B}^{d-k}(0;1)\right)$, for some $m\geq 1$. 
Next note that $\lambda_m(U\times B^{d-k}(0;1)) \subseteq \lambda_{m+1}(U\times {B}^{d-k}(0;1))$ by (ii). Therefore implementing (i) again we get
\[
\lambda_m(\{q\}\times {B}^{d-k}(0;1)) \subseteq \lambda_{m+1}(\{q\}\times {B}^{d-k}(0;1)). 
\]
\end{proof}

This finishes the proof that $M$ has a local weak vector bundle structure of rank $(d-k)$. Note that statement (iii) is superfluous in this proof. However we use (iii) to prove Proposition \ref{incr union}.

\section*{Proof of Corollary \ref{main cor}} 
\begin{proof}
If the Kobayashi metric vanishes identically at some point $p\in M$, then from Theorem \ref{main thm}, it follows that the   Kobayashi metric vanishes identically on $M$ and $M$ can be written as an increasing union of unit balls. Then the result follows from \cite [Theorem 1.2] {FP}.  

If the Kobayashi metric does not vanish in each tangent direction at some point $p\in M$, then by Theorem \ref{main thm}, the same is true for all points of $M$ and $M$ has a bounded realization. Thus the result follows immediately. 

Now suppose that the corank $k_p$ of the Kobayashi metric at some point $p\in M$ lies strictly in between $0$ and $d$. Let $k_p=d-k$. Then by  Theorem \ref{main thm}, it follows that $k_z \equiv k_p=d-k$ for all $z\in M$.  To complete the proof for this case, we continue using the same notations as in the proof of part (III) of Theorem \ref{main thm}. Let $0\not\equiv f \in L^2(M)$.  Then $f(p)\neq 0$ for some $p\in M$.  Construct $\psi:M \rightarrow B(0;1)$ with $\psi(p)=0$.  Also one can choose $q_0\in \psi(\tilde{M}_0)$ to be $0$. 
 Thus $p\in \psi^{-1}(U)$.  Further, we have 
 \[
\psi^{-1}(U) =\bigcup_{m=1}^\infty \lambda_m\left(U\times {B}^{d-k}(0;1)\right).   
\]
{\it Claim:} $L^2(\psi^{-1}(U))=\{0\}$, which in turn gives $L^2(M)=\{0\}$.  

\medskip 
\noindent
Inspired by \cite [Lemma 4.6] {FS}, we prove the following proposition. 
\begin{proposition}\label{id van Kob}
Let $q\in U$, then $\psi^{-1}\{q\}$ is an increasing union of $(d-k)$-dimensional unit ball. Further, the $(d-k)$-dimensional Kobayashi metric is identically vanishing on $\psi^{-1}\{q\}$. 
\end{proposition} \label{incr union}
\begin{proof}
It follows from (iv) that $\psi^{-1}\{q\}$ is an increasing union of $(d-k)$-dimensional unit ball. Note that by (iv) $\lambda_{m+1}(q,0)=\lambda_m(q,0)$, for all $m\geq 1$. So to show that the $(d-k)$-dimensional Kobayashi metric is identically vanishing on $\psi^{-1}\{q\}$, by Step 1 of the subsection 4.1, it is sufficient to prove that at the point $\lambda_m(q,0)$ the $(d-k)$-dimensional Kobayashi metric is vanishing. 

Let $\Delta$ be the unit disc.  Let $\gamma: \Delta \rightarrow \lambda_m(\{q\}\times {B}^{d-k}(0;1))$ be a holomorphic curve such that $\gamma(0)=\lambda_m(q,0)$ and $\gamma'(0)=Rv$ with $R>0$ and $v\neq 0$.  Consider the map $\lambda_{m+1}\circ A \circ \lambda_{m+1}^{-1} \circ \gamma: \Delta \rightarrow \lambda_{m+1}(\{q\}\times {B}^{d-k}(0;1))$, where $A: \{q\}\times {B}^{d-k}(0;{1}/{m}) \rightarrow \{q\}\times {B}^{d-k}(0;1)$ defined by $(q,z)\mapsto (q,mz)$. Then note that $(\lambda_{m+1}\circ A \circ \lambda_{m+1}^{-1} \circ \gamma) (0)=\lambda_{m+1}(0)$. Further
since $\lambda_{m+1}(q,0)=\lambda_m(q,0)$, 
\begin{eqnarray*}
(\lambda_{m+1}\circ A \circ \lambda_{m+1}^{-1} \circ \gamma)'(0) &=& \lambda_{m+1}'(q,0)A'(q,0) {\left(\lambda_{m+1}^{-1}\right)}'(\lambda_{m+1}(q,0))\gamma'(0)\\
&=& mRv.
\end{eqnarray*} 
Thus the $(d-k)$-dimensional Kobayashi metric is identically vanishing on $\psi^{-1}\{q\}$.
\end{proof}

 \subsection*{Proof of the claim:}
Let $0\not \equiv f\in L^2(\psi^{-1}(U))$. Thus $\int_{\psi^{-1}(U)} {\lvert f \rvert}^2 < \infty$, which in turn gives
$$
\int_{\lambda_m \left(U\times {B}^{d-k}(0;1)\right)} {\lvert f \rvert}^2 < \infty,
$$
for all $m\geq 1$.
Therefore, 
\begin{equation} \label{L2 int Kob}
\int_{U\times {B}^{d-k}(0;1)} {\lvert f \circ \lambda_m \rvert}^2 {\lvert {\rm{Jac}}(\lambda_m)\rvert}^2< \infty,
\end{equation}
for all $m\geq 1$.

\noindent
Note that $\lambda_m(0)=\phi_m \circ \theta_m^{-1}(0)$.
Now since $\alpha_m(0)=\psi\circ \phi_m(0)=\psi(p)=0$ and  $f_i(0)=0$ for $1\leq i \leq (d-k)$, it follows that $\theta_m(0)=0$. Thus  $\lambda_m(0)=\phi_m \circ \theta_m^{-1}(0)=\phi_m(0)=p$ and  $f\circ \lambda_m(0)=f(p)\neq 0$. 

Now we show that  $\lvert {\rm{Jac}} \lambda_m(0)\rvert$ can be taken as large as we want for sufficiently large $m$.
First note that by (IV)  we get
\begin{equation} 
\psi^{-1}\{0\}= \bigcup_{m=1}^\infty \lambda_m \left(\{0\} \times {B}^{d-k}(0;1) \right).
\end{equation}
Further by Proposition \ref{id van Kob} it follows that the $(d-k)$-dimensional Kobayashi metric vanishes identically on $\psi^{-1}\{0\}$.  Since $\lambda_m(0)=p$, it follows that $p\in \lambda_m \left(\{0\}\times {B}^{d-k}(0;1)\right)$, for all $m\geq 1$. Let $R>0$ and let $\zeta\in T_p(\psi^{-1}\{0\})$ with $\lVert\zeta \rVert=1$. Then using the compactness of the unit sphere there exists $m$ independent of $\zeta$ and there exists $\zeta_m: \Delta \rightarrow \lambda_m \left(\{0\} \times {B}^{d-k}(0;1)\right)$ such that  $\zeta_m(0)=p$ and $\zeta_m'(0)=R\zeta$. 
Now $\lambda_m^{-1}\circ \zeta_m: \Delta \rightarrow \{0\} \times {B}^{d-k}(0;1)$ for all $m$. Thus by Schwarz lemma
\begin{align*}
& \lVert (\lambda_m^{-1} \circ \zeta_m)'(0)\rVert \leq 1\\
 \Rightarrow &\; \lVert {(\lambda_m^{-1})}'(p) \zeta_m'(0)\rVert \leq 1\\
\Rightarrow & \; \lVert {(\lambda_m^{-1})}'(p) R \zeta \rVert \leq 1\\
\Rightarrow & \; \lVert {(\lambda_m^{-1})}'(p) \zeta \rVert \leq 1/ R.
\end{align*}
This gives $\lvert{\rm{Jac}}(\lambda_m^{-1})(p)\rvert \rightarrow 0$ as $m\rightarrow \infty$. Thus $\lvert{\rm{Jac}}(\lambda_m)(0)\rvert \rightarrow \infty$. Now from (\ref{L2 int Kob}), it follows that there exists some $K>0$ such that
 \begin{equation*} 
K{\lvert f(p)\rvert}^2 {\lvert{\rm{Jac}}(\lambda_m)(0)\rvert}^2 \leq \int_{U\times {B}^{d-k}(0;1)} {\lvert f \circ \lambda_m \rvert}^2 {\lvert {\rm{Jac}}(\lambda_m)\rvert}^2< \infty.
\end{equation*}
This is a contradiction. Thus $f\equiv 0$. 
\end{proof} 

\section*{Proof of Corollary \ref{Behrens}}
\begin{proof}
Suppose that the Kobayashi corank of $M$ is zero at some point. Then by (II) of Theorem \ref{main thm} it follows that $M$ is Kobayashi hyperbolic.  Now suppose that for each each $n\geq 1$,  $\psi_n: M_n \rightarrow \Omega$ is a biholomorphism. If there exists a $z\in M$ such that $\psi_n(z)$ has a limit point in $\Omega$, then by  \cite [Lemma 3.1] {MV}, $M\cong \Omega$. If there exists no such point in $M$ as above, then $\{\psi_n(z)\}_{n\geq 1}$ accumulates on $\partial \Omega$ for all $z\in M$.  Now the squeezing function of a $C^2$-smooth strictly pseudoconvex domain approaches to $1$ near the boundary (see \cite [Theorem 1.3] {DGZ1}).  Further, $M$ has a bounded realization. Thus by  \cite [Theorem 2.1] {DGZ1}, we have
\[
s_M(z)=\lim_{n\rightarrow\infty}s_{M_n}(z)=\lim_{n\rightarrow\infty} S_\Omega(\psi_n(z))=1.
\]
Therefore, $M$ is biholomorphic to the unit ball $B(0;1)$ (See \cite [Theorem 2.1] {DGZ2}). 

Now suppose that the corank $k_p$ of the Kobayashi metric at some point $p\in M$ lies strictly in between $0$ and $d$. Then by  Theorem \ref{main thm}, it follows that $k_z \equiv k_p=d-k$ for all $z\in M$. Let for each $m\geq 1$, $g_n: M_n \rightarrow \Omega$ be a biholomorphism. If there exists a $p\in M$ such that $\{g_n(p)\}_{n\geq 1}$ is compactly contained in $\Omega$, then using results in the Section 4 of \cite{FS} we get that $M$ has a weak vector bundle structure of rank $d-k$ over a submanifold of $\Omega$. Now suppose that for all $p\in M$ and for any sequence of biholomorphism $g_n: M_n \rightarrow \Omega$, the sequence $\{g_n(p)\}$ approaches to the boundary. Then $S_\Omega(g_n(p)) \rightarrow 1$ as $n\rightarrow \infty$. Consequently, in the proof of  [(III), Theorem \ref{main thm}], we can take the domain of the map $\alpha$ to be $B(0;1)$. Therefore, $M$ has a weak vector bundle structure of rank $d-k$ over a submanifold of $B(0;1)$.
\end{proof}

\textbf{Conflict of interest:} The authors state that there is no conflict of interest. No data sets were generated or analyzed during the current study. 

{\textbf{Acknowledgements:}} The second named author is partially supported by Start-up Research Grant (SRG/2023/001676) and Mathematical Research Impact Centric Support (MTR/2023/001258) from 
Science and Engineering Research Board of India.

\bibliographystyle{amsplain}

\end{document}